\documentclass[11pt]{amsart}
\usepackage[top=2cm,bottom=2cm,left=1.5cm,right=1.5cm,marginparwidth=1.75cm]{geometry}

\usepackage[square,compress,comma,numbers,sort]{natbib}
\usepackage[colorlinks=true, citecolor=red, linkcolor=blue]{hyperref}
\usepackage{amsfonts,mathtools,mathabx,amssymb}

\usepackage[shortlabels]{enumitem}

\usepackage{amsthm}
\usepackage{cleveref}
\usepackage{aliascnt}
\usepackage{xcolor}

\newtheorem{theorem}{Theorem}[section]

\newaliascnt{lemma}{theorem}
\newtheorem{lemma}[lemma]{Lemma}
\aliascntresetthe{lemma}
\crefname{lemma}{Lemma}{Lemmas}
\Crefname{lemma}{Lemma}{Lemmas}

\newaliascnt{proposition}{theorem}
\newtheorem{proposition}[proposition]{Proposition}
\aliascntresetthe{proposition}
\crefname{proposition}{Proposition}{Propositions}
\Crefname{proposition}{Proposition}{Propositions}

\newaliascnt{corollary}{theorem}

\aliascntresetthe{corollary}
\crefname{corollary}{Corollary}{Corollaries}
\Crefname{corollary}{Corollary}{Corollaries}

\theoremstyle{definition}
\newaliascnt{definition}{theorem}

\aliascntresetthe{definition}
\crefname{definition}{Definition}{Definitions}
\Crefname{definition}{Definition}{Definitions}

\newaliascnt{example}{theorem}
\newtheorem{example}[example]{Example}
\aliascntresetthe{example}
\crefname{example}{Example}{Examples}
\Crefname{example}{Example}{Examples}

\newaliascnt{assumption}{theorem}

\aliascntresetthe{assumption}
\crefname{assumption}{Assumption}{Assumptions}
\Crefname{assumption}{Assumption}{Assumptions}

\theoremstyle{remark}
\newaliascnt{remark}{theorem}
\newtheorem{remark}[remark]{Remark}
\aliascntresetthe{remark}
\crefname{remark}{Remark}{Remarks}
\Crefname{remark}{Remark}{Remarks}

\allowdisplaybreaks[4]

\DeclareMathOperator*{\sgn}{sgn}
\DeclareMathOperator*{\Cov}{Cov}
\DeclareMathOperator*{\Var}{Var}
\DeclareMathOperator*{\argmin}{argmin}

\providecommand{\supp}{\operatorname{supp}}

\newcommand{\Pp}{\mathbb P}
\newcommand{\Ee}{\mathbb E}
\newcommand{\cH}{\mathcal H}
\newcommand{\cM}{\mathcal M_1}

\title{High Minima of Gaussian Processes: Overshoots and Minimizer Locations}

\author{Enkelejd Hashorva}
\address{Enkelejd Hashorva, Department of Actuarial Science,
University of Lausanne, UNIL-Dorigny, 1015 Lausanne, Switzerland}
\email{Enkelejd.Hashorva@unil.ch}

\author{Svyatoslav Novikov}
\address{Svyatoslav Novikov, Department of Actuarial Science,
University of Lausanne, UNIL-Dorigny, 1015 Lausanne, Switzerland}
\email{Svyatoslav.Novikov@unil.ch}

\date{\today}


\begin{document}

\begin{abstract}
Let $X(t)$, $t\in K$, be a centred Gaussian process with continuous sample
paths on a compact metric space $K$, and let
$M=\min_{t\in K}X(t)$. Let $\sigma_*^2$ denote the minimum covariance
energy associated with $X$, and assume that $\sigma_*^2>0$.  Motivated by the results of
\cite{chakrabarty2018asymptotic} for smooth Gaussian processes, we show that,
conditionally on $M>u$, the scaled overshoot $u(M-u)$ converges, as
$u\to\infty$, to an exponential random variable with mean $\sigma_*^2$.
Moreover, every weak subsequential limit of the conditional law of a
measurable minimizer of $X$ is an optimal covariance-energy measure.  In
particular, if this measure is unique, then the conditional law converges
weakly to it.  The results are illustrated by stationary Gaussian processes,
fractional Brownian motion, and fractional Brownian sheet.
\end{abstract}

\maketitle

\section{Introduction}
Let
$\boldsymbol X=(X_1,\ldots,X_n)^{\mathsf T}$ be a centred Gaussian vector
with positive definite covariance matrix $\Sigma$.  Write $\boldsymbol1$
for the all-ones vector and set
\[
 K=\{1,\ldots,n\},\qquad
 M=\min_{i\in K}X_i,
 \qquad
 \sigma_*^2=\min_{\boldsymbol p\geq0,
             \,\boldsymbol1^{\mathsf T}\boldsymbol p=1}
             \boldsymbol p^{\mathsf T}\Sigma\boldsymbol p.
\]
Let $\boldsymbol p$ be the unique minimizer and define
\[
 I=\{i\in K:p_i>0\},\qquad
 E=\{i\in K:(\Sigma\boldsymbol p)_i=\sigma_*^2\},\qquad
 P=E\setminus I,\qquad N=K\setminus E.
\]
For a vector $\boldsymbol v$ and $A,B\subset K$, write
$\boldsymbol v_A=(v_i)_{i\in A}$,
$\Sigma_{AB}=(\Sigma_{ij})_{i\in A,j\in B}$, and
$\boldsymbol1_A=(1)_{i\in A}$.  Vector inequalities are understood
componentwise.
The active set $I$ is necessarily nonempty and moreover in view of  
\cite[Lemma~4.1]{debicki2020extremes}
\[ 
 \boldsymbol\theta=\Sigma_{II}^{-1}\boldsymbol1_I
 =\frac{\boldsymbol p_I}{\sigma_*^2}>0,
 \qquad
 \boldsymbol r_N=\Sigma_{NI}\Sigma_{II}^{-1}\boldsymbol1_I
 >\boldsymbol1_N.
\]
In particular, we have 
\[
 \sum_{i\in I}\theta_i=\frac1{\sigma_*^2}.
\]
If $P$ is nonempty, let $G_P$ be a centred
Gaussian random vector with covariance
\[
 \Gamma_P=\Sigma_{PP}-\Sigma_{PI}\Sigma_{II}^{-1}\Sigma_{IP},
 \qquad c_P=\Pp\{G_P>0\},
\]
and set $c_P=1$ if $P$ is empty.  We write $\mathcal L(V)$ for the law
of a random element $V$.  Throughout, let $(\xi_i)_{i\geq1}$ be
independent unit exponential random variables.  If $P$ is nonempty, let
$G_P^+$ have law $\mathcal L(G_P\mid G_P>0)$ and be independent of
$(\xi_i)_{i\geq1}$.  \\
The Gaussian specialization of
\cite[Example~6.1 and Theorem~5.1]{Hashorva2007} presented  in
Theorem~\ref{thm:finite-gaussian} below, gives as $u\to\infty$
\begin{equation}
 \Pp\{M>u\}\sim
 \frac{c_P}{(2\pi)^{|I|/2}\det(\Sigma_{II})^{1/2}
              \prod_{i\in I}\theta_i}
 u^{-|I|}\exp\left\{-\frac{u^2}{2\sigma_*^2}\right\}
 \label{eq:intro-finite-tail}
\end{equation}
and the following conditional convergence in distribution:
\begin{equation}
 \left(
   (u(X_i-u))_{i\in I},
   (X_j-u)_{j\in P},
   \frac{\boldsymbol X_N}{u}
 \right)
 \Longrightarrow
 \left(
   (\xi_i/\theta_i)_{i\in I},
   G_P^+,
   \boldsymbol r_N
 \right), \qquad u\to \infty.
 \label{eq:intro-finite-excess}
\end{equation}
The $G_P^+$ and $\boldsymbol r_N$ blocks are omitted
when $P$ and $N$, respectively, are empty.

If $T$ denotes  the almost surely unique index attaining
$\min_{i\in K}X_i$, then  
\begin{equation}
\label{eA}
 \mathcal L\bigl(u(M-u)\mid M>u\bigr)
 \Longrightarrow \mathcal L(\sigma_*^2\xi_1),
 \qquad
 \mathcal L(T\mid M>u)\Longrightarrow
 \sum_{i\in I}p_i\delta_i, \qquad u\to\infty.
\end{equation}
Here $\delta_i$ denotes the unit point mass at $i$.
The almost-sure uniqueness follows from
$\Var(X_i-X_j)>0$ for $i\ne j$.
Both limits in \eqref{eA} follow from \eqref{eq:intro-finite-excess}.
Indeed, the passive-contact and noncontact coordinates are asymptotically
larger than the active excesses on the scale $u^{-1}$:
\[
 \min_{j\in P}u(X_j-u)\longrightarrow+\infty,
 \qquad
 \min_{j\in N}u(X_j-u)\longrightarrow+\infty, \quad u\to\infty
\] 
in conditional probability, with the corresponding assertion omitted
for an empty set.  For the first limit this follows because $G_P^+$ is
coordinatewise strictly positive almost surely; for the second it follows
from $\boldsymbol r_N>\boldsymbol1_N$.  Consequently, with conditional
probability tending to one, both $M$ and $T$ are determined by the active
coordinates.  Let
\[
 S_*=\min_{i\in I}\frac{\xi_i}{\theta_i},
 \qquad
 T_*=\argmin_{i\in I}\frac{\xi_i}{\theta_i}.
\]
The limiting coordinates are independent and have continuous laws, so
the minimum and argmin mappings are almost surely continuous at the limit
vector.  Moreover, for $x\geq0$,
\[
 \Pp\{S_*>x\}
 =\prod_{i\in I}e^{-\theta_i x}
 =\exp\left\{-\frac{x}{\sigma_*^2}\right\},
 \qquad
 \Pp\{T_*=i\}
 =\frac{\theta_i}{\sum_{j\in I}\theta_j}=p_i.
\]
Thus $S_*\stackrel{d}{=}\sigma_*^2\xi_1$, which proves \eqref{eA}; see also
\cite[Proposition~3.3(ii)]{chakrabarty2018asymptotic} for a direct
derivation when $P$ is empty.

Exact asymptotics for the infimum of a centred Gaussian process
$X(t)$, $t\in K=[a,b]$, were first derived for smooth $X$ in
\cite{chakrabarty2018asymptotic}.  In particular, Theorems~4.3 and~4.4 of
\cite{chakrabarty2018asymptotic} give the exact
asymptotics of $\Pp\{M>u\}$ and the weak convergence of both the
overshoot $u(M-u)$ and the conditional law of the location $T$ of the
minimum given $M>u$, under smoothness and further regularity assumptions.
These findings rely on properties of the optimal measure $\nu$ attaining
the minimum covariance energy $\sigma_*^2>0$.

We are interested in the more general case of a continuous centred
Gaussian process $X(t)$, $t\in K$, where $K$ is a compact metric space.
While the exact asymptotics of $\Pp\{M>u\}$ in this general setting are
studied in a companion contribution \cite{debickiHashorvaNovikov2026},
here we extend both weak-convergence results mentioned above.  Our proofs
use the reproducing kernel Hilbert space (RKHS) of $X$ and the associated
covariance-energy problem.\\

Organisation of the paper: In Section~2 we introduce the setup and state
the main results.  Section~3 contains the proofs.  The appendix gives the
tail asymptotics of Gaussian random vectors and the explicit optimal
measure for fractional Brownian motion.

\section{Main results}
 
Let $K$ be a nonempty compact metric space and let
$X(t),t\in K$ be a centred Gaussian process, regarded as a measurable
$C(K)$-valued random element.  Write
\[
 R(s,t)=\Ee\{X(s)X(t)\}, \qquad s,t\in K,
\]
for its covariance kernel.  The function $R$ is continuous; this follows,
for example, from path continuity and Fernique's theorem.  Here and below,
$\|x\|_\infty=\sup_{t\in K}|x(t)|$ for $x\in C(K)$, and Fernique's theorem
gives $\Ee\{\exp(\alpha\|X\|_\infty^2)\}<\infty$ for some $\alpha>0$.
Let $\cM(K)$ denote the Borel probability measures on $K$.  For every
finite signed Borel measure $\eta$ on $K$, define its covariance energy
and potential by
\[
 \mathcal E_R(\eta)=\iint_{K^2}R(s,t)\,\eta(ds)\eta(dt),
 \qquad
 k_\eta(t)=\int_KR(s,t)\,\eta(ds).
\]
Set
\[
 \sigma_*^2=\min_{\mu\in\cM(K)}\mathcal E_R(\mu),
 \qquad
 M=\min_{t\in K}X(t).
\]
The minimum in the definition of $\sigma_*^2$ is attained by weak
compactness and continuity.  Any measure attaining it is called an
optimal measure.
The measurable maximum theorem and the Kuratowski--Ryll-Nardzewski
selection theorem provide a Borel map $S:C(K)\to K$ such that
$S(x)\in\argmin_{t\in K}x(t)$ for every $x\in C(K)$.
We let $T=S(X)$, or take any other measurable minimizer selection.

Our first result is the overshoot limit, which is a generalisation of \cite[Theorem 4.3]{chakrabarty2018asymptotic} to non-smooth Gaussian processes.

\begin{theorem} 
\label{thm:overshoot}
If   $\sigma_*^2>0$,  then as $u\to\infty$ we have the following convergence in distribution
\[
 \mathcal L\bigl(u(M-u)\mid M>u\bigr)
 \Longrightarrow \mathcal L(\sigma_*^2\xi_1).
\] 
\end{theorem}
If $\sigma_*^2=0$ and $\nu$ is an optimal measure, then
\[
 \int_KX(t)\nu(dt)=0
\]
almost surely.  Since $M\leq\int_KX(t)\nu(dt)$, it follows that
$\Pp\{M>u\}=0$ for every $u\geq0$.  The next lemma gives several
equivalent conditions for the positivity of $\sigma_*^2$.

\begin{lemma}
\label{lem:positivity}
The following conditions are equivalent:
\begin{align*}
 \textnormal{(i)}\quad &\mathcal E_R(\mu)>0\quad\text{for every }\mu\in\cM(K);
 &\textnormal{(ii)}\quad &\sigma_*^2>0;\\
 \textnormal{(iii)}\quad
 &\Pp\{M>u_0\}>0\quad\text{for some }u_0\geq0;
 &\textnormal{(iv)}\quad
 &\Pp\{M>u\}>0\quad\text{for every }u\in\mathbb R.
\end{align*}
\end{lemma}

Consequently, if $\sigma_*^2>0$, then Lemma~\ref{lem:positivity} gives
$\Pp\{M>u\}>0$ for every $u$, so all conditional laws appearing here  
are well defined.

Our second result extends \cite[Theorem 4.4]{chakrabarty2018asymptotic} to non-smooth Gaussian processes.  It shows that the conditional law of the location of the minimum converges to the unique optimal measure, if it exists.

\begin{theorem}
\label{thm:location}
Suppose that $\sigma_*^2>0$, and let $\nu$ be any optimal measure.
Every weak subsequential limit $\mu$ of $\mathcal L(T\mid M>u)$ along
$u\to\infty$ satisfies
\[
 k_\mu=k_\nu
\]
and is itself optimal.  Consequently, if the optimal measure $\nu$ is
unique, then, as $u\to\infty$, weakly in $\cM(K)$
% , equivalently against test
% functions in $C(K)$,
\begin{equation}
 \mathcal L(T\mid M>u)\Longrightarrow\nu.
 \label{eq:location}
\end{equation}
\end{theorem}

\begin{remark}  Convergence of the conditional law
of a particular minimizer selection identifies its weak limit, but does
not imply uniqueness of the optimal covariance-energy measure.

Indeed, let $K=[0,1]$ and $X(t)=Y$ for every $t\in K$, where
$Y\sim N(0,1)$.  Then $R(s,t)=1$ for all $s,t\in K$, and therefore every
probability measure on $K$ is optimal.  The leftmost selector satisfies
\[
 \mathcal L(T\mid M>u)=\delta_0
\]
for every $u\in\mathbb R$, whereas the rightmost selector has law
$\delta_1$.  Thus convergence of a particular selector does not imply
uniqueness of the optimal measure, and the subsequential optimality
conclusion is the sharp selector-independent statement in this case.
\end{remark}

\begin{remark}
\label{rem:stationary-setting}

% Pointwise uniqueness of the sample minimizer is not needed.  For example,
% on $K=\{0,1,2\}$ let $X_0=X_1=Z+G$ and $X_2=Z$, with $Z,G$ independent
% standard normal variables.  The unique optimal measure is $\delta_2$,
% although the sample minimum is tied at $0$ and $1$ whenever $G<0$.
% As $u\to\infty$, the theorem nevertheless gives convergence of the
% leftmost minimizer to $\delta_2$ under the high-minimum conditioning.

In the stationary setting of \cite[Theorem~4.1]{wu2019high}, suppose
separately that the covariance-energy problem on $[0,L]$ has a unique
optimal measure $\nu$ and that $\sigma_*^2>0$.  Then
Theorem~\ref{thm:location} gives, for any measurable selection $T$ from
the sample minimizers,
\[
 \mathcal L\left(
 T\ \middle|\ \min_{0\leq t\leq L}X(t)>u
 \right)
 \Longrightarrow \nu,
 \qquad u\to\infty.
\] 
A convenient sufficient condition for both assumptions is that the support
of the spectral measure of $X$ contains a nonempty open interval; see
\cite[Proposition~2.1]{muirheadSevero2024percolation}.
\end{remark}

We next present three examples that illustrate our findings above.
 
\begin{example}[Stationary Ornstein--Uhlenbeck process]
Let $X$ be centred stationary with covariance
\[
 R(s,t)=e^{-\lambda|t-s|},\qquad s,t\in[0,L],\quad \lambda,L>0.
\]
Then the unique optimal measure and its energy are
\[
 \nu=\frac{\delta_0+\delta_L+\lambda\,dt}{2+\lambda L},
 \qquad
 \sigma_*^2=\frac{2}{2+\lambda L}.
\]
Here $dt$ denotes Lebesgue measure on $[0,L]$.
Indeed, for every $t\in[0,L]$ we have 
\[
 k_\nu(t)=\frac{e^{-\lambda t}+e^{-\lambda(L-t)}
 +\lambda\int_0^L e^{-\lambda|t-s|}\,ds}{2+\lambda L}
 =\frac{2}{2+\lambda L}.
\]
Thus $\nu$ is optimal, and it is unique since the kernel is integrally
strictly positive definite.  Consequently, as $u\to\infty$ the following weak convergences hold:
\[
 \mathcal L\bigl(u(M-u)\mid M>u\bigr)
 \Longrightarrow
 \mathcal L\left(\frac{2}{2+\lambda L}\xi_1\right),
 \qquad
 \mathcal L(T\mid M>u)\Longrightarrow
 \frac{\delta_0+\delta_L+\lambda\,dt}{2+\lambda L}.
\]
\end{example}

\begin{example}[Fractional Brownian motion]
\label{ex:fbm}
Let $B_H$ be a standard fractional Brownian motion on $[a,b]$, where
$0<a<b<\infty$ and $H\in(0,1)$, with covariance
\[
 R_H(s,t)=\frac12\bigl(s^{2H}+t^{2H}-|t-s|^{2H}\bigr), \qquad s,t\ge 0.
\]
The unique optimal measure $\nu_H$ is identified in
Lemma~\ref{lem:fbm-optimal} below.  If $H\geq1/2$, then
\[
 \nu_H=\delta_a,\qquad \sigma_H^2=a^{2H}.
\]
For $x,y>0$ and $z\in[0,1]$, write
\[
 \mathrm B(x,y)=\int_0^1v^{x-1}(1-v)^{y-1}\,dv,\qquad
 \mathrm B_z(x,y)=\int_0^zv^{x-1}(1-v)^{y-1}\,dv,
\]
and $\mathrm I_z(x,y)=\mathrm B_z(x,y)/\mathrm B(x,y)$ for the beta,
incomplete beta, and regularised incomplete beta functions, respectively.
If $0<H<1/2$, put $\kappa=1/2-H$ and
\[
 A_H=\frac{(b-a)^{1-2\kappa}}{\mathrm B(\kappa,\kappa)}
    \mathrm I_{a/b}(1-\kappa,2\kappa),\qquad
 D_H=\frac{\sin(\pi \kappa)}{\pi}\frac{a^{1-\kappa}}{b^\kappa}.
\]
Then
\[
 \nu_H(dt)=\rho_H(t)\,dt,
\]
where
\[
 \rho_H(t)=(t-a)^{\kappa-1}(b-t)^{\kappa-1}
 \left(A_H+D_H\frac{b-t}{t}\right),\qquad a<t<b.
\]
Moreover,
\[
 \inf_{a<t<b}\rho_H(t)>0,\qquad \supp(\nu_H)=[a,b],
\]
and its potential is constant:
\[
 \sigma_H^2=\int_a^bR_H(s,t)\rho_H(t)\,dt>0,
 \qquad s\in[a,b].
\]
Writing $M_H=\min_{a\leq t\leq b}B_H(t)$ and taking any measurable
sample minimizer $T_H$, Theorems~\ref{thm:overshoot}
and~\ref{thm:location} give, as $u\to\infty$,
\[
 \mathcal L\bigl(u(M_H-u)\mid M_H>u\bigr)
 \Longrightarrow \mathcal L(\sigma_H^2\xi_1),
 \qquad
 \mathcal L(T_H\mid M_H>u)\Longrightarrow\nu_H.
\]
The condition $a>0$ is essential: if $0\in K$, then $B_H(0)=0$,
$\sigma_*^2=0$, and the high-minimum event is empty for $u\geq0$.
\end{example}

\begin{example}[Fractional Brownian sheet]
Let $\mathbb R_+=[0,\infty)$, and let
$W_H=\{W_H(\boldsymbol t),\boldsymbol t\in\mathbb R_+^2\}$ be a
continuous fractional Brownian sheet with Hurst index $H\in(0,1)$, so that
\[
 \Ee\bigl\{W_H(\boldsymbol s)W_H(\boldsymbol t)\bigr\}
 =R_H(s_1,t_1)R_H(s_2,t_2),
 \qquad \boldsymbol s,\boldsymbol t\in\mathbb R_+^2.
\]
Consider $W_H$ on $K=[a,b]^2$, where $0<a<b<\infty$, and let
$\nu_H$ and $\sigma_H^2$ be as in Example~\ref{ex:fbm}.  Then the unique
optimal measure and its energy are
\[
 \nu_H^{\otimes2}=\nu_H\otimes\nu_H,\qquad
 \sigma_{H,\square}^2=(\sigma_H^2)^2.
\]
Indeed, the energy and potential factorise:
\[
 \mathcal E_R(\nu_H^{\otimes2})=(\sigma_H^2)^2,
 \qquad
 k_{\nu_H^{\otimes2}}(\boldsymbol t)
 =k_{\nu_H}(t_1)k_{\nu_H}(t_2)
 \geq(\sigma_H^2)^2.
\]
Thus Theorem~\ref{thm:optimal-measure-criterion} gives optimality.
The product kernel is also integrally strictly positive definite.  Indeed,
if a finite signed measure $\eta$ on $[a,b]^2$ has zero product energy,
then its product potential vanishes.  Applying the last assertion of
Lemma~\ref{lem:fbm-optimal} first in the first coordinate gives
\[
 \int_{A\times[a,b]}R_H(s_2,t_2)\,\eta(d\boldsymbol s)=0
\]
for every Borel set $A\subset[a,b]$ and every $t_2\in[a,b]$; applying it
again in the second coordinate gives $\eta(A\times B)=0$ for all Borel
rectangles, and hence $\eta=0$.  Therefore the optimal measure is unique.

Writing
\[
 M_{H,\square}=\min_{\boldsymbol t\in[a,b]^2}W_H(\boldsymbol t)
\]
and taking any measurable sample minimizer $T_{H,\square}$, we obtain
\[
 \mathcal L\bigl(u(M_{H,\square}-u)\mid M_{H,\square}>u\bigr)
 \Longrightarrow
 \mathcal L\bigl((\sigma_H^2)^2\xi_1\bigr),
 \qquad
 \mathcal L(T_{H,\square}\mid M_{H,\square}>u)
 \Longrightarrow\nu_H^{\otimes2}.
\]
For $H\geq1/2$ the limiting location measure is $\delta_{(a,a)}$,
whereas for $0<H<1/2$ it has density
$\rho_H(t_1)\rho_H(t_2)$ on $(a,b)^2$.
\end{example}

\section{Proofs}
\label{sec:proofs}
We first present the RKHS identities for the covariance energy and
potential defined above.  Let $\cH$ be the RKHS of the covariance kernel
$R$, with inner product $\langle\cdot,\cdot\rangle_{\cH}$ and norm
$\|\cdot\|_{\cH}$.  Since
\[
 \|R(s,\cdot)-R(t,\cdot)\|_{\cH}^2
 =R(s,s)+R(t,t)-2R(s,t), \qquad s,t \in K
\]
the map $t\mapsto R(t,\cdot)$ is continuous from $K$ into $\cH$.
Consequently, for every finite signed measure $\eta$ on $K$, the
Bochner integral
\[
 k_\eta=\int_KR(t,\cdot)\,\eta(dt)
\]
belongs to $\cH$.  Define the isonormal map first on kernel sums by
\[
 W\left(\sum_{j=1}^nc_jR(t_j,\cdot)\right)
 =\sum_{j=1}^nc_jX(t_j)
\]
and extend it by isometry to $\cH$.  Approximation by kernel sums and
Fubini's theorem, justified by Fernique integrability, then give
\begin{equation}
 \begin{split}
 W(k_\eta)&=\int_KX(t)\,\eta(dt),\\
 \langle h,k_\eta\rangle_{\cH}&=\int_Kh(t)\,\eta(dt),\qquad h\in\cH,\\
 \|k_\eta\|_{\cH}^2
 &=\mathcal E_R(\eta)
 =\Ee\left\{\left(\int_KX(t)\,\eta(dt)\right)^2\right\}\geq0.
 \end{split}
 \label{eq:rkhs-measure-identities}
\end{equation}
In particular, every $h\in\cH$ is continuous, hence bounded, on $K$,
because $h(t)=\langle h,R(t,\cdot)\rangle_{\cH}$.

We shall use the following standard optimality criterion for minimum-energy
probability measures; see   \cite[Theorem~4.3]{adler2014existence}, \cite[Theorem~2.1]{selk2024large} and
\cite[Theorem~1.1]{pronzato2021minimum}.  We include its short proof for
completeness.

\begin{theorem}
\label{thm:optimal-measure-criterion} The probability measure
\(\mu\in\cM(K)\) is optimal
if and only if

\[
        k_\mu(t)\ge \mathcal E_R(\mu),
        \qquad t\in K.
\]

Moreover, if \(\mu\) is optimal, then
\[
        k_\mu(t)=\mathcal E_R(\mu),
        \qquad t\in\supp(\mu).
\]

In particular, if \(k_\mu\) is constant on \(K\), then \(\mu\) is
optimal. If, in addition,
\[
        \mathcal E_R(\eta)>0
\]
for every nonzero finite signed measure \(\eta\) on \(K\) satisfying
\(\eta(K)=0\), then the optimal measure is unique.

Consequently, if \(\supp(\mu)=K\), then \(\mu\) is optimal if and only if
\(k_\mu\) is constant on \(K\).
\end{theorem}

\begin{proof}
For \(\nu\in\cM(K)\), put \(\eta=\nu-\mu\).  For
\(0\leq\varepsilon\leq1\),
\begin{align*}
 \mathcal E_R(\mu+\varepsilon\eta)
 &=\mathcal E_R(\mu)
   +2\varepsilon\left(\int_Kk_\mu(t)\,\nu(dt)
                         -\mathcal E_R(\mu)\right)
   +\varepsilon^2\mathcal E_R(\eta).
\end{align*}
If \(\mu\) is optimal, the right derivative at zero is nonnegative.
Taking \(\nu=\delta_t\) gives
\(k_\mu(t)\geq\mathcal E_R(\mu)\) for every \(t\in K\).  Since
\[
 \int_Kk_\mu(t)\,\mu(dt)=\mathcal E_R(\mu),
\]
equality holds \(\mu\)-almost everywhere and, by continuity, on
\(\supp(\mu)\).

Conversely, suppose that
\(k_\mu\geq\mathcal E_R(\mu)\) on \(K\).  The same expansion with
\(\varepsilon=1\) gives
\[
 \mathcal E_R(\nu)-\mathcal E_R(\mu)
 =2\left(\int_Kk_\mu(t)\,\nu(dt)-\mathcal E_R(\mu)\right)
   +\mathcal E_R(\nu-\mu)\geq0,
\]
and hence \(\mu\) is optimal.  If \(k_\mu\) is constant, its constant
value is \(\mathcal E_R(\mu)\), so the criterion applies.  Strict positivity
of \(\mathcal E_R\) on nonzero signed measures of total mass zero makes
\(\mathcal E_R\) strictly convex on \(\cM(K)\), and therefore the
optimal measure is unique.  The last assertion follows from the equality
on \(\supp(\mu)\).
\end{proof}

\begin{proof}[Proof of Lemma~\ref{lem:positivity}]
The implication $\Pp\{M>0\}>0\Rightarrow\sigma_*^2>0$ is noted after
(1.2) in \cite{wu2019high}; related RKHS and minimum-energy dual
formulations appear in \cite[Theorem~5.1]{adler2014existence}.

If $\mu_n\Rightarrow\mu$ weakly in $\cM(K)$, then
$\mu_n\otimes\mu_n\Rightarrow\mu\otimes\mu$.  Since $R$ is continuous
and bounded on the compact space $K^2$, it follows that
$\mathcal E_R(\mu_n)\to\mathcal E_R(\mu)$.  Thus weak compactness of
$\cM(K)$ gives (i)$\Leftrightarrow$(ii).  If (iii) holds, then for every
$\mu\in\cM(K)$ the centred Gaussian variable
$\int_KX(t)\mu(dt)$ is strictly positive with positive probability,
and hence its variance $\mathcal E_R(\mu)$ is positive.  Thus
(iii)$\Rightarrow$(i), while (iv)$\Rightarrow$(iii) is immediate.

It remains to prove (ii)$\Rightarrow$(iv).  Let $\nu$ be optimal and put
\[
 h=\frac{k_\nu}{\sigma_*^2}\in\cH.
\]
Theorem~\ref{thm:optimal-measure-criterion} gives $h\geq1$ on $K$.
Choose $L>0$ such that $\Pp\{\|X\|_\infty<L\}>0$, and then choose
$r>\max\{0,u+L\}$.  On $\{\|X\|_\infty<L\}$,
\[
 \min_{t\in K}\{X(t)+rh(t)\}>r-L>u.
\]
The Cameron--Martin theorem says that the laws of $X+rh$ and $X$ are
equivalent.  Hence $\Pp\{M>u\}>0$, proving (iv).
\end{proof}

We now prove Theorems~\ref{thm:overshoot} and~\ref{thm:location}.  The
following regression estimates will be used in both proofs.

\subsection{Regression and logarithmic concentration}
\label{subsec:regr}

Fix an arbitrary optimal measure $\nu$.  The optimality criterion gives
\begin{equation}
 k_\nu(t)\geq\sigma_*^2 \quad(t\in K),
 \qquad
 k_\nu(t)=\sigma_*^2 \quad(t\in\supp(\nu)).
 \label{eq:frostman}
\end{equation}
Set
\[
 Y=\int_KX(t)\,\nu(dt),
 \qquad
 m(t)=\frac{k_\nu(t)}{\sigma_*^2},
 \qquad
 Z(t)=X(t)-m(t)Y.
\]
Then $Y\sim N(0,\sigma_*^2)$, $m\geq1$, and Gaussian regression gives
\begin{equation}
 X(t)=m(t)Y+Z(t),
 \label{eq:regression}
\end{equation}
where $Z$ is a centred continuous Gaussian process independent of $Y$.

The law of $M$ is absolutely continuous.  For $z\in C(K)$, write
\[
 f_z(y)=\min_{t\in K}\{m(t)y+z(t)\},\qquad y\in\mathbb R.
\]
The bound
\[
 |f_z(y)-f_{\widetilde z}(\widetilde y)|
 \leq \|z-\widetilde z\|_\infty
      +\|m\|_\infty|y-\widetilde y|
\]
shows that $(z,y)\mapsto f_z(y)$ is jointly continuous and hence Borel
measurable.  Since $C(K)$ is Polish, regular conditional distributions
given $Z$ exist; independence makes the conditional law of $Y$ its
original Gaussian law.  Now, conditionally on $Z$, since
$1\leq m\leq\|m\|_\infty$, for $y_2>y_1$ we have
\[
 y_2-y_1\leq f_Z(y_2)-f_Z(y_1)
 \leq\|m\|_\infty(y_2-y_1).
\]
Thus, for every fixed realisation of $Z$, the map $f_Z$ is a strictly
increasing bi-Lipschitz homeomorphism of $\mathbb R$, and its inverse is
$1$-Lipschitz.  If $B\subset\mathbb R$ has Lebesgue measure zero, then
$f_Z^{-1}(B)$ also has Lebesgue measure zero.  Since $Y$ has a density
and is independent of $Z$,
\[
 \Pp\{M\in B\mid Z\}
 =\Pp\{Y\in f_Z^{-1}(B)\mid Z\}=0
\]
almost surely.  Hence the law of $M=f_Z(Y)$ is absolutely continuous,
and in particular atomless.

The regression also gives a short proof of the logarithmic tail estimate.
Since $M>u$ implies $Y>u$, while for every $\delta>0$ and all
sufficiently large $u$,
\[
 \{Y>(1+2\delta)u,\ \|Z\|_\infty<\delta u\}
 \subseteq\{M>u\},
\]
independence yields
\begin{equation}
 \lim_{u\to\infty}u^{-2}\log\Pp\{M>u\}
 =-\frac1{2\sigma_*^2},
 \label{eq:ldp}
\end{equation}
which is also obtained in \cite[Theorem~5.1]{adler2014existence}.

It follows as well that, conditionally on $M>u$,
\begin{equation}
 \frac Yu\longrightarrow1,
 \qquad
 \frac{\|Z\|_\infty}{u}\longrightarrow0
 \quad\text{in probability}
 \label{eq:concentration}
\end{equation}
as $u\to\infty$. 
Indeed, on the conditioning event $Y/u>1$, while for every
$\varepsilon>0$,
\[
 \Pp\{Y>(1+\varepsilon)u\mid M>u\}
 \leq \frac{\Pp\{Y>(1+\varepsilon)u\}}{\Pp\{M>u\}}
 \longrightarrow0
\]
as $u\to\infty$ exponentially fast by the Gaussian tail and \eqref{eq:ldp}.
For the second assertion, put
$\sigma_Z^2=\sup_{t\in K}\Var(Z(t))$.  If $\sigma_Z^2=0$, then
$Z\equiv0$ almost surely and the assertion is immediate.  Otherwise,
Borell--TIS gives, for every $\varepsilon>0$,
\[
 \log\Pp\{\|Z\|_\infty>\varepsilon u\}
 \leq-\frac{\varepsilon^2u^2}{2\sigma_Z^2}+o(u^2).
\]
Since $M>u\Rightarrow Y>u$, independence gives
\[
 \frac{\Pp\{\|Z\|_\infty>\varepsilon u,M>u\}}{\Pp\{M>u\}}
 \leq
 \Pp\{\|Z\|_\infty>\varepsilon u\}
 \frac{\Pp\{Y>u\}}{\Pp\{M>u\}}
 \leq e^{-c_\varepsilon u^2+o(u^2)},
\]
by the preceding bound and \eqref{eq:ldp}, where
$c_\varepsilon=\varepsilon^2/(2\sigma_Z^2)>0$.  For $u>0$, the event
$\{M>u\}$ implies $Y>u>0$.
Since $M=X(T)\leq Y$, \eqref{eq:regression} gives
\[
 0\leq (m(T)-1)Y
 \leq -Z(T)
 \leq \|Z\|_\infty,
\]
Consequently, as $u\to\infty$
\begin{equation}
 m(T)\longrightarrow1
 \quad\text{in probability under }\Pp\{\,\cdot\mid M>u\}.
 \label{eq:contact}
\end{equation}

\subsection{The overshoot}

\begin{proof}[Proof of Theorem~\ref{thm:overshoot}]
The key observation replacing all smooth local
analysis is that $q$ is log-concave.  Indeed, let $\gamma$ be the Gaussian
law of $X$ on $C(K)$ and define the closed convex sets
\[
 \mathcal A_u=\left\{x\in C(K):\min_{t\in K}x(t)\geq u\right\},
 \qquad u\in\mathbb R.
\]
Every Gaussian measure on a separable Banach space is log-concave, and
\[
 \lambda\mathcal A_u+(1-\lambda)\mathcal A_v
 \subseteq\mathcal A_{\lambda u+(1-\lambda)v},\qquad 0\leq\lambda\leq1.
\]
Consequently, $u\mapsto\gamma(\mathcal A_u)=\Pp\{M\geq u\}$ is log-concave.  The
law of $M$ is atomless, so this function equals $q(u)=\Pp\{M>u\}$.  This is also the
Gaussian-measure form of the Pr\'ekopa theorem~\cite{Prekopa1973}.

For completeness, the same fact can be verified by finite-dimensional
approximation, including the degenerate case.  Choose a dense sequence
$(t_j)_{j\geq1}$ in $K$ and put
\[
 q_n(u)=\Pp\{X(t_j)\geq u,\ 1\leq j\leq n\}.
\]
Let $r_n$ be the rank of the covariance matrix of
$(X(t_1),\ldots,X(t_n))$.  Choose $B_n\in\mathbb R^{n\times r_n}$ so
that this vector has the same law as $B_n\boldsymbol G$, where
$\boldsymbol G$ is standard Gaussian in $\mathbb R^{r_n}$; the rank-zero
case is immediate.  With $\mathbf1_A$ denoting the indicator of $A$, the
function
\[
 (\boldsymbol z,u)\longmapsto
 \mathbf1_{\{B_n\boldsymbol z\geq u\boldsymbol1\}}
 \exp\{-\|\boldsymbol z\|^2/2\}
\]
is log-concave because its support is convex.  The Pr\'ekopa marginal
theorem therefore shows that each $q_n$ is log-concave.  The defining
events decrease with $n$, and path continuity gives
$\inf_{j\geq1}X(t_j)=M$.  Hence
\[
 q_n(u)\downarrow\Pp\{M\geq u\}=q(u)
\]
and the defining log-concavity inequality passes to the limit.

Lemma~\ref{lem:positivity} gives $q(u)>0$ for every $u\in\mathbb R$.

Set $\psi(u)=-\log q(u)$.  Thus $\psi$ is convex and \eqref{eq:ldp}
states that, as $u\to\infty$
\begin{equation}
 \frac{\psi(u)}{u^2}\longrightarrow
 c:=\frac1{2\sigma_*^2}.
 \label{eq:psi-rate}
\end{equation}
For completeness, convexity alone upgrades this quadratic rate to the
required local ratio.  Fix $x>0$ and $\varepsilon\in(0,1)$.  With
$\Delta_u=x/u$, whenever $u\geq\sqrt{x/\varepsilon}$ we have
$\Delta_u\leq\varepsilon u$, so the monotonicity of convex secant slopes
gives
\[
 \frac{\psi(u)-\psi((1-\varepsilon)u)}{\varepsilon u}
 \leq \frac{\psi(u+\Delta_u)-\psi(u)}{\Delta_u}
 \leq \frac{\psi((1+\varepsilon)u)-\psi(u)}{\varepsilon u}.
\]
Multiplying by $\Delta_u=x/u$, using \eqref{eq:psi-rate}, and first letting
$u\to\infty$ gives
\[
 c(2-\varepsilon)x
 \leq\liminf_{u\to\infty}\{\psi(u+x/u)-\psi(u)\}
 \leq\limsup_{u\to\infty}\{\psi(u+x/u)-\psi(u)\}
 \leq c(2+\varepsilon)x.
\]
Now let $\varepsilon\downarrow0$.  Thus, for every $x\geq0$, as $u\to\infty$
\begin{equation}
 \psi\left(u+\frac xu\right)-\psi(u)
 \longrightarrow2cx=\frac{x}{\sigma_*^2}.
 \label{eq:local-tail}
\end{equation}
Consequently, as $u\to\infty$,
\[
 \Pp\{u(M-u)>x\mid M>u\}
 =\frac{q(u+x/u)}{q(u)}
 \longrightarrow e^{-x/\sigma_*^2}.
\]
For $x<0$, both the conditional survival function (for all sufficiently
large $u$) and that of $\sigma_*^2\xi_1$ equal one;
hence the convergence for $x\geq0$ proves the asserted weak convergence establishing the claim.  
\end{proof}

\subsection{A Cameron--Martin tie lemma}

We shall use the following observation: for every fixed $h\in\cH$,
\begin{equation}
 h \text{ is constant on }\argmin_{t\in K}X(t)
 \quad\text{almost surely}.
 \label{eq:tie-lemma}
\end{equation}
To prove it, for $x\in C(K)$ define
\[
 \phi_x(r)=\min_{t\in K}\{x(t)+r h(t)\},\qquad r\in\mathbb R.
\]
The map $(x,r)\mapsto\phi_x(r)$ is jointly continuous, and for each $x$
the function of $r$ is concave and Lipschitz.  Danskin's
formula~\cite{Danskin1966} gives
\[
 \phi'_{x,+}(r)
 =\min_{t\in\argmin_{s\in K}\{x(s)+rh(s)\}}h(t),
 \qquad
 \phi'_{x,-}(r)
 =\max_{t\in\argmin_{s\in K}\{x(s)+rh(s)\}}h(t).
\]
These formulas are elementary in the present affine perturbation: for
$s>0$, comparison with minimizers at $r$ and at $r+s$ bounds the
difference quotient of $\phi_x$ between the corresponding values of
$h$, and compactness lets $s\downarrow0$; the left derivative is handled
similarly.
Thus $h$ is nonconstant on the minimizer set of $x+rh$ precisely when
$\phi_x$ is not differentiable at $r$.  For each $x$ this happens at
most countably many $r$.  The corresponding set of pairs
$(x,r)\in C(K)\times\mathbb R$ is Borel: the two one-sided derivatives are limits
of measurable rational difference quotients.  Fubini's theorem therefore
shows that
\[
 \Pp\left\{h\text{ is nonconstant on }
 \argmin_{t\in K}\{X(t)+rh(t)\}\right\}=0
\]
for almost every $r$.  But the Cameron--Martin theorem says that the laws
of $X+rh$ and $X$ are equivalent for every $r$.  Choosing one of the
almost-everywhere good values of $r$, equivalence implies that the same
bad set has probability zero under the law of $X$.  This proves
\eqref{eq:tie-lemma}.  The null set may depend on $h$, which is all that
is needed below.

There is also a classical Lipschitz-functional proof.  Let $B$ be the
closed linear support of the law of $X$ in $C(K)$ and define
\[
 \Phi(x)=\min_{t\in K}x(t),\qquad x\in B.
\]
The functional $\Phi$ is $1$-Lipschitz.  The space $B$ is separable and
the law of $X$ is nondegenerate on $B$.  Hence the Gaussian Rademacher
theorem  \cite[Theorem~6]{Phelps1978} makes $\Phi$
G\^ateaux differentiable almost surely.  Since $\cH\subset B$, Danskin's
formula then gives \eqref{eq:tie-lemma}, in fact on one common
full-measure set for all $h\in\cH$.

If one assumes the stronger condition
\[
 \Var(X(s)-X(t))>0 \qquad(s\ne t),
\]
then almost-sure uniqueness of the sample minimizer follows directly from
the Gaussian argmax lemma  \cite[Lemma 2.6]{KimPollard1990},
so their result implies \eqref{eq:tie-lemma} under this condition.  It
does not cover the general version used here, because its hypothesis
excludes precisely the degenerate increments that can produce tied
minimizers.

\subsection{The location of the minimum}

\begin{proof}[Proof of Theorem~\ref{thm:location}]
Fix $h\in\cH$ and define
\[
 F_u(x)=\left(\min_{t\in K}x(t)-u\right)_+,
\]
where $a_+=\max\{a,0\}$ for $a\in\mathbb R$.
The Cameron--Martin formula gives
\[
 \Ee\{F_u(X+\varepsilon h)\}
 =\Ee\left\{F_u(X)
   \exp\left\{\varepsilon W(h)
   -\frac{\varepsilon^2}{2}\|h\|_{\cH}^2\right\}\right\}.
\]
By \eqref{eq:tie-lemma}, $h(T)$ has the same almost-sure value for every
measurable minimizer selection $T$.  Danskin's formula then shows that,
at a path with $M\ne u$, both one-sided derivatives of $F_u$ at
$\varepsilon=0$ equal
\[
 \mathbf1_{\{M>u\}}h(T),
\]
for any minimizer selection $T$.  The exceptional event $\{M=u\}$ has
probability zero by the atomlessness proved in \Cref{subsec:regr}.  The difference
quotients of $\varepsilon\mapsto F_u(X+\varepsilon h)$ on the left-hand
side are bounded by $\|h\|_\infty$, which is finite by
\eqref{eq:rkhs-measure-identities} and the paragraph following it.  On
the right-hand side,
$F_u(X)\in L^2$ by Fernique's theorem
and $W(h)$ has all exponential moments.  Dominated differentiation is
valid.  We use the notation
$\Ee\{U;A\}=\Ee\{U\mathbf1_A\}$ for an event $A$.  Thus
\begin{equation}
 \Ee\{h(T);M>u\}=\Ee\{W(h)(M-u)_+\}.
 \label{eq:ibp}
\end{equation}

The identities \eqref{eq:rkhs-measure-identities} with $\eta=\nu$ give
\[
 W(k_\nu)=Y,
 \qquad
 \|k_\nu\|_{\cH}^2=\sigma_*^2,
 \qquad
 \langle h,k_\nu\rangle_{\cH}=\int_Kh\,d\nu.
\]
Taking $h=k_\nu$ in \eqref{eq:ibp} gives the exact normalisation
\begin{equation}
 \Ee\{Y(M-u)_+\}=\Ee\{k_\nu(T);M>u\}.
 \label{eq:normalization}
\end{equation}

For general $h\in\cH$, write
\[
 h=\alpha k_\nu+g,
 \qquad
 \alpha=\frac{\int_Kh\,d\nu}{\sigma_*^2},
 \qquad
 \langle g,k_\nu\rangle_{\cH}=0.
\]
Then $G=W(g)$ is independent of $Y$, and \eqref{eq:ibp}--\eqref{eq:normalization}
give
\begin{equation}
 \Ee\{h(T);M>u\}
 =\alpha\Ee\{k_\nu(T);M>u\}+\Ee\{G(M-u)_+\}.
 \label{eq:decomposition}
\end{equation}
We claim that, as $u\to\infty$, the last term is $o(q(u)), q(u)=\Pp\{M>u\}$.  Put
$A_u=(M-u)_+$.  On all
outcomes
\[
 0\leq A_u\leq(Y-u)_+.
\]
On $\{M>u\}$ we have $Y>u$.  Equation
\eqref{eq:normalization} therefore gives, for $u>0$,
\[
 u\Ee\{A_u\}\leq\Ee\{YA_u\}
 \leq\|k_\nu\|_\infty q(u).
\]
For every $\varepsilon>0$, independence of $G$ and $Y$ yields
\[
 \Ee\{|G|A_u\}
 \leq\varepsilon u\Ee\{A_u\}
 +\Ee\{|G|;|G|>\varepsilon u\}\,\Ee\{(Y-u)_+\}.
\]
If $\tau^2=\Var G>0$, standard Gaussian tail estimates give
\[
 \log\Ee\{|G|;|G|>\varepsilon u\}
 \leq-\frac{\varepsilon^2u^2}{2\tau^2}+o(u^2),
 \qquad
 \log\Ee\{(Y-u)_+\}
 =-\frac{u^2}{2\sigma_*^2}+o(u^2).
\]
Thus \eqref{eq:ldp} makes the product in the second term $o(q(u))$; if
$\tau=0$, it vanishes.
It follows that
\[
 \limsup_{u\to\infty}
 \frac{|\Ee\{GA_u\}|}{q(u)}
 \leq\varepsilon\|k_\nu\|_\infty.
\]
Letting $\varepsilon\downarrow0$ proves the claim.

Since $m$ is continuous on the compact space $K$,
$1\leq m(T)\leq\|m\|_\infty$.  Thus \eqref{eq:contact} and uniform
boundedness imply convergence of the conditional expectations:
\[
 \Ee\{m(T)\mid M>u\}\longrightarrow1, \qquad u\to\infty.
\]
Because $k_\nu(T)=\sigma_*^2m(T)$, dividing
\eqref{eq:decomposition} by $q(u)$ and using the preceding display and
the $o(q(u))$ estimate now gives, as $u\to\infty$,
\begin{equation}
 \Ee\{h(T)\mid M>u\}\longrightarrow\int_Kh\,d\nu,
 \qquad h\in\cH.
 \label{eq:H-tests}
\end{equation}
The conditional laws of $T$ are tight because $K$ is compact.  If $\mu$
is a weak subsequential limit, \eqref{eq:H-tests} gives
\[
 \int_Kh\,d\mu=\int_Kh\,d\nu,
 \qquad h\in\cH.
\]
Taking $h=R(t,\cdot)\in\cH$ for each $t\in K$ shows that the kernel mean
embeddings agree: $k_\mu=k_\nu$.  Hence
\[
 \mathcal E_R(\mu)=\|k_\mu\|_{\cH}^2
       =\|k_\nu\|_{\cH}^2=\sigma_*^2,
\]
so $\mu$ is optimal.  If the optimiser is unique, $\mu=\nu$; compactness
then gives the full convergence \eqref{eq:location}.

Finally, the potential $k_\nu$ does not depend on which optimal measure
is chosen.  Indeed, if $\nu_1$ and $\nu_2$ are optimal, convexity and
minimality show that $(\nu_1+\nu_2)/2$ is also optimal.  The parallelogram
identity then gives
\[
 \mathcal E_R(\nu_1-\nu_2)
 =\|k_{\nu_1}-k_{\nu_2}\|_{\cH}^2=0,
\]
and hence $k_{\nu_1}=k_{\nu_2}$.  This also confirms that the conclusion
is independent of the initially fixed optimal measure $\nu$.
\end{proof}

The following optional subsection is not needed for the proofs of
\Cref{thm:overshoot,thm:location}.  It is included because it gives the
selection-independent Malliavin derivative $DM=R_T$, even when sample
minimizers are tied, and yields a reusable covariance-transfer identity
for functions of the minimum.  In particular, it recovers \eqref{eq:ibp}
by taking $f(x)=(x-u)_+$.

\subsection{A Malliavin location and covariance-transfer identity}

The classical Malliavin differentiability argument for extrema goes back
to Nualart and Vives~\cite[Theorems~1--2]{NualartVives1988}; see also
\cite[Propositions~2.1.10--2.1.11]{Nualart2006}.  In the unique-maximizer
fractional-Brownian setting, an explicit derivative at the maximizing
location appears in \cite[Lemma~3]{LanjriZadiNualart2003}.  The identity
below gives the corresponding RKHS location form in the present Gaussian
setting, including tied minimizers; Gaussian duality then gives the
covariance-transfer formula.  For $t\in K$, write
$R_t=R(t,\cdot)\in\cH$, so that $X(t)=W(R_t)$, and let $D$ denote the
Malliavin derivative associated with the isonormal process $W$.  We write
$\mathbb D^{1,2}$ for the corresponding first-order Malliavin--Sobolev
space.

\begin{proposition} 
\label{prop:malliavin-location}
Suppose that $\sigma_*^2>0$, and let $T$ be any measurable selection from
$\argmin_{t\in K}X(t)$.  Then $M\in\mathbb D^{1,2}$ and
\begin{equation}
 DM=R_T:=R(T,\cdot)\qquad\text{almost surely}.
 \label{eq:derivative-minimum}
\end{equation}
The random element $R_T\in\cH$ is independent of the chosen minimizer
selection.  Moreover, if $f:\mathbb R\to\mathbb R$ is locally absolutely
continuous and $f'\in L^\infty(\mathbb R)$, then
$f(M)\in\mathbb D^{1,2}$ and, for every $h\in\cH$,
\begin{equation}
 \Ee\bigl\{W(h)f(M)\bigr\}
 =\Ee\bigl\{f'(M)h(T)\bigr\}.
 \label{eq:malliavin-location}
\end{equation}
Here $f'$ may be chosen arbitrarily on its Lebesgue-null set of
non-differentiability.
\end{proposition}

\begin{proof}
The Hilbert space $\cH$ is separable because $K$ is compact metric,
$t\mapsto R_t$ is continuous from $K$ to $\cH$, and the linear span of
$\{R_t:t\in K\}$ is dense in $\cH$.  Apply \eqref{eq:tie-lemma} to a
countable dense subset of $\cH$ and intersect the resulting full-measure
events.  Since
\[
 \sup_{t\in K}|h(t)|
 \leq\left(\sup_{t\in K}\|R_t\|_{\cH}\right)\|h\|_{\cH},
 \qquad h\in\cH,
\]
density shows that, on this common event, every $h\in\cH$ is constant on
the sample minimizer set.  If $s$ and $t$ are both sample minimizers,
then
\[
 \langle h,R_s-R_t\rangle_{\cH}=h(s)-h(t)=0,
 \qquad h\in\cH,
\]
and hence $R_s=R_t$.  Thus $R_T$ does not depend on the selection.

Let $(t_j)_{j\geq1}$ be dense in $K$, put
\[
 M_n=\min_{1\leq j\leq n}X(t_j),
\]
let $J_n$ be the smallest index in
$\argmin_{1\leq j\leq n}X(t_j)$, and put $T_n=t_{J_n}$.
The finite-dimensional Malliavin chain rule gives
\[
 DM_n=R_{T_n}\qquad\text{almost surely}.
\]
Indeed, a tie between $t_i$ and $t_j$ either has
$R_{t_i}=R_{t_j}$, in which case the two possible derivatives agree, or
$X(t_i)-X(t_j)=W(R_{t_i}-R_{t_j})$ is nondegenerate, in which case the
tie has probability zero.

Path continuity gives $M_n\to M$ almost surely and in $L^2$.  Every
cluster point of $(T_n)$ is a sample minimizer.  Continuity of
$t\mapsto R_t$ and the selection independence proved above therefore
give $R_{T_n}\to R_T$ almost surely in $\cH$.  The random elements
$R_{T_n}$ are uniformly bounded in $\cH$, so this convergence also holds
in mean square with values in $\cH$.  Closedness of $D$ proves
\eqref{eq:derivative-minimum}.

The absolute continuity of the law of $M$ proved above implies that $f$
is differentiable at $M$ almost surely.  The Malliavin chain rule gives
$Df(M)=f'(M)R_T$.  Gaussian integration by parts for deterministic
$h\in\cH$ now yields
\[
 \Ee\{W(h)f(M)\}
 =\Ee\{\langle h,Df(M)\rangle_{\cH}\}
 =\Ee\{f'(M)h(T)\},
\]
which proves \eqref{eq:malliavin-location}.
\end{proof}

Since $W(h)$ is centred, \eqref{eq:malliavin-location} has the
covariance-transfer form
\begin{equation}
 \Cov\bigl(W(h),f(M)\bigr)=\Ee\bigl\{f'(M)h(T)\bigr\}.
 \label{eq:covariance-transfer}
\end{equation}
In particular, for $s\in K$ and every finite signed measure $\eta$ on
$K$,
\[
 \Cov\bigl(X(s),f(M)\bigr)=\Ee\bigl\{f'(M)R(s,T)\bigr\},
\]
and
\[
 \Cov\left(\int_KX(t)\,\eta(dt),f(M)\right)
 =\Ee\bigl\{f'(M)k_\eta(T)\bigr\}.
\]
Taking $f(x)=x$ gives $\Cov(X(s),M)=\Ee\{R(s,T)\}$.  Taking instead
$f(x)=(x-u)_+$ recovers \eqref{eq:ibp}; with
$\mu_u=\mathcal L(T\mid M>u)$, it also gives the exact
identity
\[
 k_{\mu_u}(s)
 =\frac{\Ee\{X(s)(M-u)_+\}}{q(u)},\qquad s\in K.
\]
 
% \begin{remark}[Relation to the 2018 proof]
% The differentiability and curvature condition (4.6) in
% Chakrabarty--Samorodnitsky~\cite{chakrabarty2018asymptotic} are needed for their precise tail constant and
% their local Taylor analysis.  They are not needed for either conclusion
% above.  The overshoot follows from log-concavity plus the quadratic tail
% rate, while the location follows from Gaussian integration by parts.
% \end{remark}

\appendix
\section{Auxiliary results}

\subsection{Finite-dimensional Gaussian asymptotics}

We present below the Gaussian specialization of
\cite[Example~6.1, Theorem~5.1]{Hashorva2007}. 

\begin{theorem}[Gaussian tail and conditional excess limit]
\label{thm:finite-gaussian}
Let $\boldsymbol X$ be a centred Gaussian vector in $\mathbb R^n$ with
positive definite covariance matrix $\Sigma$, and let $\boldsymbol p$ be
the minimizer of the following problem:
\[
 \sigma_*^2=\min_{\boldsymbol v\geq0,
             \,\boldsymbol1^{\mathsf T}\boldsymbol v=1}
             \boldsymbol v^{\mathsf T}\Sigma\boldsymbol v.
\]
Put
\[
 I=\{i:p_i>0\},\qquad J=\{1,\ldots,n\}\setminus I,
 \qquad  
 \qquad \boldsymbol\theta=\Sigma_{II}^{-1}\boldsymbol1_I.
\]
Then
\[
 \boldsymbol\theta=\frac{\boldsymbol p_I}{\sigma_*^2}>0,
 \qquad
 \boldsymbol r_J=\Sigma_{JI}\Sigma_{II}^{-1}\boldsymbol1_I\geq\boldsymbol1_J,
 \qquad
 \boldsymbol1_I^{\mathsf T}\boldsymbol\theta=\frac1{\sigma_*^2}.
\]
Write $r_j$ for the $j$th component of $\boldsymbol r_J$.
Let
\[
 P=\{j\in J:r_j=1\},\qquad N=\{j\in J:r_j>1\},
\]
and let $G_P$ be centred Gaussian with covariance matrix
\[
 \Gamma_P=\Sigma_{PP}-\Sigma_{PI}\Sigma_{II}^{-1}\Sigma_{IP}.
\]
Set $c_P=\Pp\{G_P>0\}$, with $c_P=1$ when $P$ is empty.  If
$M=\min_{1\leq i\leq n}X_i$, then, as $u\to\infty$,
\begin{equation}
 \Pp\{M>u\}\sim
 \frac{c_P}{(2\pi)^{|I|/2}\det(\Sigma_{II})^{1/2}
                  \prod_{i\in I}\theta_i}
 u^{-|I|}\exp\left\{-\frac{u^2}{2\sigma_*^2}\right\}.
 \label{eq:finite-gaussian-tail}
\end{equation}
Let $(\xi_i)_{i\geq1}$ be independent unit exponential random variables.
If $P$ is nonempty, let $G_P^+$ have law
$\mathcal L(G_P\mid G_P>0)$ and be independent of $(\xi_i)_{i\geq1}$.
Furthermore, as $u\to\infty$,
\begin{equation}
 \mathcal L\left(
 \left(
   (u(X_i-u))_{i\in I},
   (X_j-u)_{j\in P},
   \frac{\boldsymbol X_N}{u}
 \right)
   \ \middle|\ M>u\right)
 \Longrightarrow
 \mathcal L\left(
   (\xi_i/\theta_i)_{i\in I},
   G_P^+,
   \boldsymbol r_N
 \right).
 \label{eq:finite-gaussian-excess}
\end{equation}
The $G_P^+$ and $\boldsymbol r_N$ blocks are omitted
when $P$ and $N$, respectively, are empty.
\end{theorem}
 
\subsection{The optimal measure for fractional Brownian motion}

\begin{lemma}
\label{lem:fbm-optimal}
Let $B_H$ and $\nu_H$ be as in Example~\ref{ex:fbm}.  Then $\nu_H$ is
the unique optimal measure for the covariance-energy problem on $[a,b], 0< a<b<\infty$.
For $0<H<1/2$, its density $\rho_H$ satisfies
\[
 \inf_{a<t<b}\rho_H(t)>0,\qquad \supp(\nu_H)=[a,b],
\]
and its potential is constant and positive on $[a,b]$.
For every $H\in(0,1)$, the covariance kernel $R_H$ is integrally strictly
positive definite on $[a,b]$.
\end{lemma}

\begin{proof}
Suppose first that $H\geq1/2$.  For $t\geq a$,
\[
 R_H(a,t)-R_H(a,a)
 =\frac12\bigl(t^{2H}-a^{2H}-(t-a)^{2H}\bigr)\geq0,
\]
since $x\mapsto x^{2H}$ is superadditive on $[0,\infty)$.  Thus
$\delta_a$ is optimal by Theorem~\ref{thm:optimal-measure-criterion},
and its energy is $a^{2H}$.

Let now $0<H<1/2$, put $\kappa=1/2-H$, $d=b-a$, and write
\[
 \rho_1(t)=(t-a)^{\kappa-1}(b-t)^{\kappa-1},\qquad
 \rho_2(t)=(t-a)^{\kappa-1}(b-t)^{\kappa}t^{-1}.
\]
Thus $\rho_H=A_H\rho_1+D_H\rho_2$.  This function is integrable and positive
on $(a,b)$.  Indeed, $0<\kappa<1/2$, the only endpoint exponents are
$\kappa-1$ and $\kappa$, and $t^{-1}$ is bounded on $[a,b]$.  Since $A_H,D_H>0$,
$\rho_H$ is continuous and positive on $(a,b)$ and tends to infinity
at both endpoints.  Choose $\varepsilon>0$ so that $\rho_H\geq1$ on
$(a,a+\varepsilon)\cup(b-\varepsilon,b)$.  Its minimum on
$[a+\varepsilon,b-\varepsilon]$ is also positive, and therefore
$\inf_{a<t<b}\rho_H(t)>0$.

We first check the normalisation.  If
$\iota=\mathrm I_{a/b}(1-\kappa,2\kappa)$, the beta integral and the substitution
$x=b(t-a)/(td)$ give
\begin{align*}
 \int_a^b\rho_1(t)\,dt
 &=d^{2\kappa-1}\mathrm B(\kappa,\kappa),\\
 \int_a^b\rho_2(t)\,dt
 &=\frac{\pi}{\sin(\pi \kappa)}a^{\kappa-1}b^\kappa(1-\iota).
\end{align*}
For completeness, the inverse substitution and its differential are
\[
 t=\frac{ab}{b-dx},\qquad
 dt=\frac{abd}{(b-dx)^2}\,dx.
\]
Consequently,
\begin{align*}
 \int_a^b\rho_2(t)\,dt
 &=a^{\kappa-1}b^{-\kappa}d^{2\kappa}\mathrm B(\kappa,\kappa+1)
 {}_2F_1\left(2\kappa,\kappa;2\kappa+1;\frac db\right).
\end{align*}
Here ${}_2F_1$ denotes the Gauss hypergeometric function.
Now
\[
 \mathrm B_z(2\kappa,1-\kappa)
 =\frac{z^{2\kappa}}{2\kappa}{}_2F_1(2\kappa,\kappa;2\kappa+1;z),
 \qquad \mathrm B(\kappa,\kappa+1)=\frac12\mathrm B(\kappa,\kappa),
\]
and
\[
 \kappa\mathrm B(\kappa,\kappa)\mathrm B(2\kappa,1-\kappa)=\frac{\pi}{\sin(\pi \kappa)},
\]
while the beta-complement identity gives
$\mathrm I_{1-a/b}(2\kappa,1-\kappa)=1-\iota$.  These identities give the second
integral displayed above.  Hence the two terms in
$\int_a^b\rho_H(t)\,dt$ are $\iota$ and $1-\iota$, respectively.
Thus $\nu_H(dt)=\rho_H(t)\,dt$ is a probability measure and
$\supp(\nu_H)=[a,b]$.

Consider its potential
\[
 U(y)=\int_a^bR_H(y,t)\rho_H(t)\,dt.
\]
It is continuous on $[a,b]$.  To justify differentiation, fix a compact
interval $J_0\subset(a,b)$.  The density is bounded near every point of
$J_0$, and, uniformly for $y\in J_0$,
\[
 \int_{|t-y|<\varepsilon}|t-y|^{-2\kappa}\,dt
 =O(\varepsilon^{1-2\kappa}).
\]
Away from the diagonal, the endpoint singularities of $\rho_H$ are
integrable.  Hence $U\in C^1((a,b))$, and splitting at $t=y$ and
differentiating the two improper integrals is legitimate.  Since
$\int_a^b\rho_H(t)\,dt=1$ and $2H-1=-2\kappa$, this gives
\begin{equation}
 \frac{U'(y)}H
 =y^{-2\kappa}-A_H\{L_1(y)-Q_1(y)\}-D_H\{L_2(y)-Q_2(y)\},
 \label{eq:fbm-potential-derivative}
\end{equation}
where
\[
 L_i(y)=\int_a^y(y-t)^{-2\kappa}\rho_i(t)\,dt,\qquad
 Q_i(y)=\int_y^b(t-y)^{-2\kappa}\rho_i(t)\,dt.
\]
After the substitutions $t=a+(y-a)x$ and $t=y+(b-y)x$, respectively,
Euler's beta integral gives
\begin{equation}
 L_1(y)=Q_1(y)
 =\mathrm B(\kappa,1-2\kappa)d^{2\kappa-1}
   (y-a)^{-\kappa}(b-y)^{-\kappa}.
 \label{eq:fbm-first-beta-identity}
\end{equation}
We give a short calculation of the remaining difference.  With
\[
 x=\frac{b(t-a)}{td},\qquad z=\frac{b(y-a)}{yd},
\]
we have
\[
 t=\frac{ab}{b-dx},\qquad
 y-t=\frac{dyt}{ab}(z-x).
\]
It follows directly that
\begin{align*}
 L_2(y)
 &=a^{\kappa-1}b^{\kappa}y^{-2\kappa}
   \int_0^z(z-x)^{-2\kappa}x^{\kappa-1}(1-x)^\kappa\,dx,\qquad 
 Q_2(y)
 &=a^{\kappa-1}b^{\kappa}y^{-2\kappa}
   \int_z^1(x-z)^{-2\kappa}x^{\kappa-1}(1-x)^\kappa\,dx
\end{align*} implying
\begin{equation}
 L_2(y)-Q_2(y)=a^{\kappa-1}b^{\kappa}y^{-2\kappa}\Delta_\kappa(z),
 \label{eq:fbm-second-beta-reduction}
\end{equation}
where
\[
 \Delta_\kappa(z)=\int_0^1\sgn(z-x)|z-x|^{-2\kappa}x^{\kappa-1}(1-x)^\kappa\,dx.
\]
Put $w(x)=x^{\kappa-1}(1-x)^{\kappa-1}$ and
\[
 J(z)=\int_0^1\sgn(z-x)|z-x|^{-2\kappa}w(x)\,dx,\qquad
 V(z)=\int_0^1|z-x|^{1-2\kappa}w(x)\,dx.
\]
Since $1-x=(1-z)+(z-x)$,
\[
 \Delta_\kappa(z)=(1-z)J(z)+V(z).
\]
The same beta calculation as in \eqref{eq:fbm-first-beta-identity} gives
\[
 \int_0^z(z-x)^{-2\kappa}w(x)\,dx
 =\int_z^1(x-z)^{-2\kappa}w(x)\,dx
 =\mathrm B(\kappa,1-2\kappa)z^{-\kappa}(1-z)^{-\kappa}.
\]
Consequently $J(z)=0$.  Since $2\kappa<1$, differentiation under the integral
defining $V$ is justified by the same local integrability argument, so
$V'(z)=(1-2\kappa)J(z)=0$ and hence  
\[
 \Delta_\kappa(z)=V(0)=\mathrm B(1-\kappa,\kappa)=\frac{\pi}{\sin(\pi \kappa)}.
\]
Combining this with \eqref{eq:fbm-second-beta-reduction},
\[
 L_2(y)-Q_2(y)
 =\frac{\pi}{\sin(\pi \kappa)}\frac{b^\kappa}{a^{1-\kappa}}y^{-2\kappa}.
\]
Substitution in \eqref{eq:fbm-potential-derivative} and the definition of
$D_H$ now give $U'(y)=0$ for $a<y<b$.  Continuity at the endpoints shows
that $U$ is constant on $[a,b]$, and
Theorem~\ref{thm:optimal-measure-criterion} shows that $\nu_H$ is optimal.
Its energy is positive: if $t\geq s\geq a$, then
$R_H(s,t)\geq s^{2H}/2>0$.

It remains only to prove integral strict positive definiteness,
simultaneously for all $H\in(0,1)$.  We use the Fourier representation
 \[
 |x|^{2H}=c'_H\int_{\mathbb R}
 (1-\cos(\omega x))|\omega|^{-1-2H}\,d\omega,
 \qquad c'_H>0.
\]
Let $\eta$ be a finite signed measure on $[a,b]$, put
$m_\eta=\eta([a,b])$, and set
$\widetilde\eta=\eta-m_\eta\delta_0$.  Also write
$\widehat\eta(\omega)=\int_{[a,b]}e^{i\omega t}\eta(dt)$.  Since
$R_H(0,t)=0$, the covariance energy of $\eta$ equals that of
$\widetilde\eta$.  Moreover, $\widetilde\eta$ has total mass zero, so the
displayed identity gives, for a constant $c_H>0$,
\begin{align}
 \iint_{[a,b]^2}R_H(s,t)\,\eta(ds)\eta(dt)
 &=-\frac12\iint_{[0,b]^2}|t-s|^{2H}\,
                    \widetilde\eta(ds)\widetilde\eta(dt)\nonumber\\
 &=c_H\int_{\mathbb R}|\widehat\eta(\omega)-m_\eta|^2
                    |\omega|^{-1-2H}\,d\omega.
 \label{eq:fbm-strict-energy}
\end{align}
Because $\widetilde\eta$ has compact support and total mass zero,
$|\widehat\eta(\omega)-m_\eta|\leq C_\eta\min\{|\omega|,1\}$ for some
$C_\eta<\infty$.  Thus the
integrand in \eqref{eq:fbm-strict-energy} is $O(|\omega|^{1-2H})$ at zero
and $O(|\omega|^{-1-2H})$ at infinity, and is integrable because $0<H<1$.
If the integral vanishes, then the Fourier transform of
$\widetilde\eta$ vanishes identically, and hence $\widetilde\eta=0$.
Thus $\eta=m_\eta\delta_0$; since $\eta$ is supported on $[a,b]$ with $a>0$,
we conclude that $m_\eta=0$ and $\eta=0$.  This proves integral strict positive
definiteness, and the strict-energy part of
Theorem~\ref{thm:optimal-measure-criterion} proves uniqueness of $\nu_H$.
\end{proof}

\bibliographystyle{ieeetr}
\bibliography{GG,infArgmaxV9}

\begin{thebibliography}{10}

\bibitem{chakrabarty2018asymptotic}
A.~Chakrabarty and G.~Samorodnitsky, ``Asymptotic behaviour of high {G}aussian minima,'' {\em Stochastic Processes and their Applications}, vol.~128, no.~7, pp.~2297--2324, 2018.

\bibitem{debicki2020extremes}
K.~D{\c{e}}bicki, E.~Hashorva, and L.~Wang, ``Extremes of vector-valued {G}aussian processes,'' {\em Stochastic Processes and their Applications}, vol.~130, no.~9, pp.~5802--5837, 2020.

\bibitem{Hashorva2007}
E.~Hashorva, ``Asymptotic properties of type {I} elliptical random vectors,'' {\em Extremes}, vol.~10, pp.~175--206, 2007.

\bibitem{debickiHashorvaNovikov2026}
K.~D{\c{e}}bicki, E.~Hashorva, and S.~Novikov, ``Exact asymptotics for high infima of {G}aussian processes in finite-active regimes.'' Manuscript, 2026.

\bibitem{wu2019high}
Z.~Wu, A.~Chakrabarty, and G.~Samorodnitsky, ``High minima of non-smooth {G}aussian processes,'' {\em Electronic Communications in Probability}, vol.~24, pp.~Paper No. 53, 1--12, 2019.

\bibitem{muirheadSevero2024percolation}
S.~Muirhead and F.~Severo, ``Percolation of strongly correlated {G}aussian fields, {I}: Decay of subcritical connection probabilities,'' {\em Probability and Mathematical Physics}, vol.~5, no.~2, pp.~357--412, 2024.

\bibitem{adler2014existence}
R.~J. Adler, E.~Moldavskaya, and G.~Samorodnitsky, ``On the existence of paths between points in high level excursion sets of {G}aussian random fields,'' {\em The Annals of Probability}, vol.~42, no.~3, pp.~1020--1053, 2014.

\bibitem{selk2024large}
Z.~Selk, ``Large deviations for high minima of {G}aussian processes with nonnegatively correlated increments,'' {\em Statistics \& Probability Letters}, vol.~206, p.~110001, 2024.

\bibitem{pronzato2021minimum}
L.~Pronzato and A.~Zhigljavsky, ``Minimum-energy measures for singular kernels,'' {\em Journal of Computational and Applied Mathematics}, vol.~382, p.~113089, 2021.

\bibitem{Prekopa1973}
A.~Pr{\'e}kopa, ``On logarithmic concave measures and functions,'' {\em Acta Scientiarum Mathematicarum (Szeged)}, vol.~34, pp.~335--343, 1973.

\bibitem{Danskin1966}
J.~M. Danskin, ``The theory of max-min, with applications,'' {\em SIAM Journal on Applied Mathematics}, vol.~14, no.~4, pp.~641--664, 1966.

\bibitem{Phelps1978}
R.~R. Phelps, ``Gaussian null sets and differentiability of lipschitz map on {B}anach spaces,'' {\em Pacific Journal of Mathematics}, vol.~77, no.~2, pp.~523--531, 1978.

\bibitem{KimPollard1990}
J.~Kim and D.~Pollard, ``Cube root asymptotics,'' {\em The Annals of Statistics}, vol.~18, no.~1, pp.~191--219, 1990.

\bibitem{NualartVives1988}
D.~Nualart and J.~Vives, ``Continuit{\'e} absolue de la loi du maximum d'un processus continu,'' {\em C. R. Acad. Sci. Paris S{\'e}r. I Math.}, vol.~307, no.~7, pp.~349--354, 1988.

\bibitem{Nualart2006}
D.~Nualart, {\em The Malliavin Calculus and Related Topics}.
\newblock Berlin: Springer, 2~ed., 2006.

\bibitem{LanjriZadiNualart2003}
N.~Lanjri~Zadi and D.~Nualart, ``Smoothness of the law of the supremum of the fractional {B}rownian motion,'' {\em Electron. Commun. Probab.}, vol.~8, pp.~102--111, 2003.

\end{thebibliography}
\end{document}